\documentclass[a4paper, reqno]{amsart}

\usepackage[T1]{fontenc}
\usepackage{lmodern, microtype, amssymb}
\usepackage[scr=dutchcal]{mathalpha}
\usepackage[dvipsnames]{xcolor}
\usepackage[colorlinks=true, citecolor=RoyalBlue, linkcolor=black, urlcolor=RoyalBlue]{hyperref}
\usepackage{tikz-cd}

\microtypecontext{spacing=nonfrench}
\microtypesetup{stretch=20, shrink=15, spacing=true}
\tikzcdset{arrow style=tikz, diagrams={>={Computer Modern Rightarrow[length=5pt,width=4pt]}}}
\newtheorem{theorem}{Theorem}[section]
\newtheorem*{main-theorem}{Main theorem}
\newtheorem{proposition}[theorem]{Proposition}
\newtheorem{corollary}[theorem]{Corollary}
\newtheorem{lemma}[theorem]{Lemma}

\theoremstyle{definition}
\newtheorem*{lubin-conjecture}{Lubin's conjecture}
\newtheorem*{up}{Universal property}

\theoremstyle{remark}
\newtheorem{remark}[theorem]{Remark}

\DeclareMathOperator\bt{BT}
\DeclareMathOperator\BT{\mathbf{BT}}

\DeclareMathOperator\D{\mathbf{D}}
\DeclareMathOperator\Dsp{Dsp}
\DeclareMathOperator\End{End}
\DeclareMathOperator\FGL{\mathbf{FGL}\!}
\DeclareMathOperator\Fil{Fil}
\DeclareMathOperator\FIL{\mathbf{Fil}}
\DeclareMathOperator\Frac{Frac}

\DeclareMathOperator\Hom{Hom}
\DeclareMathOperator\Log{Log}
\DeclareMathOperator\MF{\mathbf{MF}}
\DeclareMathOperator\REP{\mathbf{Rep}}
\DeclareMathOperator\Stab{Stab}
\DeclareMathOperator\T{T}
\DeclareMathOperator\V{V}
\DeclareMathOperator\W{W}

\def\Aty{A_\infty}
\def\AinfL{A_{\mathrm{inf},L}}
\def\BinfL{B_{\mathrm{inf},L}}
\def\Acris{A_\mathrm{cris}}
\def\Bcris{B_\mathrm{cris}}
\def\BdR{B_\mathrm{dR}}

\usepackage{changepage}
\usepackage{inconsolata}

\title[Lubin's conjecture: height $1$ over $p(p-1)$-tame extensions]{Lubin's conjecture for height-one $p$-adic dynamical systems over $p(p-1)$-tame extensions}
\author{Martin Debaisieux}
\address{Département de Mathématique, Université de Mons, 7000 Mons, Belgium}
\email{martin.debaisieux@umons.ac.be}
\thanks{The author was supported by the FRIA of the Fonds de la Recherche Scientifique -- FNRS}
\subjclass{\texttt{11S82}, \texttt{11S31}, \texttt{37P20}; \texttt{14L05}, \texttt{11F80}, \texttt{11F85}}
\date{\today}

\begin{document}

\begin{adjustwidth}{-2mm}{-2mm}
    \begin{abstract}
        We prove Lubin's conjecture for height-one commuting pairs of formal power series defined over the ring of integers $\mathcal{O}$ of finite extensions of $\mathbb{Q}_p$ whose ramification index is coprime to $p(p-1)$. For such a pair $(f, u)$, with $f$ noninvertible and $u$ invertible, we show that $(f, u)$ is a pair of endomorphisms of a formal group defined over $\mathcal{O}$. The proof relies on extracting a crystalline character of weight $1$ from the Galois set $T_f$ of $f$-consistent sequences, using it to define a $\mathbb{Z}_p$-module structure on $T_f$ for which $f$ is an endomorphism, and applying explicit functors in integral $p$-adic Hodge theory to recover the latent formal group.
    \end{abstract}
\end{adjustwidth}

\begin{adjustwidth}{-2mm}{-2mm}
    \maketitle
\end{adjustwidth}

\section*{Introduction}

\noindent Pairs of analytic transformations of a $p$-adic open unit disk with a fixed point at $0$ that commute under composition were studied by Lubin in \cite{Lub94} under the name of $p$-adic dynamical systems. He pointed out in \textit{ibidem} that ‘‘for an invertible series to commute with a noninvertible series, there must be a formal group somehow in the background''. This observation has initiated a wave of results in this direction, see for instance \cite{Li96}, \cite{Li02}, \cite{Sar05}, \cite{Sar10}, \cite{SS13}, \cite{Ber19}. More recently, $p$-adic Hodge theory has proved to be an efficient way for tackling this problem, as in the work of \cite{Ber16}, \cite{Spe18}, \cite{Poy25}, and this paper is written in this vein.

\begin{lubin-conjecture}[{\cite[p. 131]{Sar05}}]
    Suppose that $f$ and $u$ are a noninvertible and a nontorsion invertible series, respectively, defined over the ring of integers~$\mathcal{O}$ of a finite extension of $\mathbb{Q}_p$. Suppose further that the roots of $f$ and all of its iterates are simple, and that $f'(0)$ is a uniformizer in $\mathbb{Z}_p$. If $f \circ u = u \circ f$, then $f, u \in \End_\mathcal{O}(F)$ for some formal group $F$ over $\mathcal{O}$.
\end{lubin-conjecture}

Throughout this article, we refer to the above statement as Lubin's conjecture. The height of the dynamical system together with its ring of definition is a common way to decompose this conjecture into cases. The case of height $1$ over $\mathbb{Z}_p$ is solved by Specter in \cite{Spe18} where he recovered the underlying Lubin-Tate power series. Shortly after, Berger solved in \cite{Ber19}  the cases of finite height over $\mathbb{Z}_p$ using $p$-adic analysis only. In this article, we solve the height-one cases over extensions whose ramification index is relatively prime to $p(p-1)$, provided that the linear coefficient of the invertible power series lies in $\mathbb{Z}_p$. The remaining cases of Lubin's conjecture are still open nowadays. Our approach consists in generalizing Specter's analysis in \cite{Spe18} of the set of consistent sequences of the dynamical~sys\-tem which yields a Galois character, recovering the missing $\mathbb{Z}_p$-module structure on it and, from this data, retrieving the latent formal group of the dynamical system via integral $p$-adic Hodge-theoretic methods.

    \subsection*{Notations} Let $p$ be a prime number and fix an algebraic closure of $\mathbb{Q}_p$ containing all our extensions. Define $\mathbb{C}_p$ to be its $p$-adic completion, with open unit disk $\mathfrak{m}_{\mathbb{C}_p\!}$. When $E$ is a finite extension of $\mathbb{Q}_p$, we let $\mathcal{O}_E$ be its ring of integers, with maximal ideal $\mathfrak{m}_E$, and we let $e_E$ be the ramification index of the extension or simply $e$ if it is unambiguous. Given a ring $R$, we define $\mathscr{S}_0(R) = XR[\![X]\!]$ to be the set of formal power series over $R$ without constant term. Composition induces a law on $\mathscr{S}_0(R)$ making it into a noncommutative monoid, and a necessary and sufficient condition on $s \in \mathscr{S}_0(R)$ to be invertible is that $s'(0)$ is a unit of $R$. We deal exclusively with rings $R$ arising from a $p$-adic situation, that is, $p$-adic fields, their integers and finite fields of characteristic $p$.

    \subsection*{Motivation and main statement} Before stating our main theorem, we remind the reader of some elementary results about formal groups, see also \cite{Haz12} for an extended presentation. By a formal group, we mean a one-dimensional commutative formal group law. Given a formal group $F$ defined over the ring of integers $\mathcal{O}_K$ of a finite extension $K$ of $\mathbb{Q}_p$, its ring of $\mathcal{O}_K$-endomorphisms, namely $\End_{\mathcal{O}_K}(F)$, is a composition-commuting family of noninvertible and invertible formal power series, provided with an embedding of $\mathbb{Z}_p$ sending $m \in \mathbb{Z}_p$ to $[m]_F$, the multiplication-by-$m$ endomorphism in $F$. The map taking the linear coefficient in $\End_{\mathcal{O}_K}(F)$ is injective. If we further assume that $F$ has height $1$ and $K$ is a nontrivial extension of $\mathbb{Q}_p$,~then $F$ is not a Lubin-Tate formal group and this is the major difference with Specter's case. What remains true is that $\End_{\mathcal{O}_K}(F)$ is a free $\mathbb{Z}_p$-algebra of rank $1$, implying that the linear coefficient of every $\mathcal{O}_K$-endomorphism of $F$ lies in $\mathbb{Z}_p$. We will refer to this fact as (\hypertarget{fact-1}{1}). An important data attached to the formal group $F$ is its $p$-adic Tate module
    \[
        \T_p(F) = \varprojlim_{x \mapsto [p]_F(x)} F[p^n]
    \]
    where $F[p^n]$ is the group of $p^n$-torsion points of $F$ in $\mathfrak{m}_{\mathbb{C}_p\!}$ for every integer $n \geqslant 0$. This is a crystalline character of the absolute Galois group $G_K$ of $K$ of Hodge-Tate weight $1$, meaning that $\T_p(F)$ is a free $\mathbb{Z}_p$-module of rank $1$ equipped with an action of $G_K$ that is crystalline of weight $1$. By \cite[Corollary 6.2.3]{SW13} what makes this object relevant is that the functor
    \[
        \begin{tabular}{lcl}
            $\T_p\colon\;$Formal Groups of Height $1$ & $\longrightarrow$ & Crystalline Characters of $G_K$\\
            \phantom{$\T_p\colon\;$}over $\mathcal{O}_K$ & & of Hodge-Tate Weight $1$
        \end{tabular}
    \]
    is an equivalence. Indeed, by \cite{Tat67} height-one formal groups correspond to height-one connected $p$-divisible groups, giving rise to crystalline characters of weight $1$, and \cite[Theorem 5.3.2]{Bre00} when $p \neq 2$ and \cite[Corollary 2.2.6]{Kis06} for any $p$ show that any such character arises from such a $p$-divisible group. Now sufficiently explicit functors to recover a formal group from the right-hand side can be extracted from the work of Kisin \cite{Kis06}, \cite{Kis09} and Zink \cite{Zin01}, \cite{Zin02} in integral $p$-adic Hodge theory.

    Let $f \in \mathscr{S}_0(\mathcal{O}_K)$ be a power series whose reduction modulo $\mathfrak{m}_K$ is nonzero, such that the roots of $f$ and all of its iterates in $\mathfrak{m}_{\mathbb{C}_p\!}$ are simple, $f$ has exactly $p$ of them and with $f'(0)$ a uniformizer in $\mathbb{Z}_p$. Motivated by the above discussion, we consider the $G_K$-set
    \[
        T_f = \varprojlim_{x \mapsto f(x)} \Lambda_n(f)
    \]
    of $f$-consistent sequences of elements in $\mathfrak{m}_{\mathbb{C}_p\!}$ (see \cite[p. 329]{Lub94}), where $\Lambda_n(f)$ is the set of roots of $f^{\circ n}$ in $\mathfrak{m}_{\mathbb{C}_p\!}$ for every integer $n \geqslant 0$. If Lubin's conjecture holds, $T_f$ should be the $p$-adic Tate module of an integral formal group once $f$ commutes with a nontorsion invertible power series $u \in \mathscr{S}_0(\mathcal{O}_K)$. Therefore there should exist a compatible $\mathbb{Z}_p$-module structure on $T_f$ turning it into a crystalline character of weight $1$ and admitting $f$ as an endomorphism. This is what we achieve to recover, for a restriction of $G_K$ from $K$ to a finite extension $L$, when $(e_K, p(p-1)) = 1$.~Based on (\hyperlink{fact-1}{1}), we assume as well that $u'(0) \in \mathbb{Z}_p^\times$. Although this assumption is reasonable, the author regrets not knowing if it can be deduced from the original hypotheses. The $\mathbb{Z}_p[G_L]$-module structure on $T_f$ allows to recover a formal group over $\mathcal{O}_L$ for which $(f, u)$ is a pair of endomorphisms. A uniqueness argument then implies that this formal group is defined over $\mathcal{O}_K$.

    \begin{main-theorem}[\ref{R: main-theorem}]
        Let $K$ be a finite extension of $\mathbb{Q}_p\!$ satisfying $(e_K, p(p-1)) = 1$. Let $(f, u)$ be a commuting pair of elements in $\mathscr{S}_0(\mathcal{O}_K)\!$ such that $f \bmod \mathfrak{m}_K \neq 0$, the roots of $f$ and all of its iterates in $\mathfrak{m}_{\mathbb{C}_p\!}$ are simple and $f$ has $p$ of them, with $f'(0)$ a uniformizer in $\mathbb{Z}_p\!$ and $u'(0) \in \mathbb{Z}_p^\times\!$ nontorsion. Then there exists a formal group $F$ defined over $\mathcal{O}_K\!$ such that $f, u \in \End_{\mathcal{O}_K}(F)$.
    \end{main-theorem}

    \subsection*{The case of small ramification (\texorpdfstring{$e < p-1$}{e < p-1})} Our work has a nontrivial intersection with the one of Berger in \cite{Ber19}. If we assume $(f, u)$ to be a commuting pair over $\mathcal{O}_K$ as in our main theorem with $(e, p(p-1)) = 1$ and that $e < p-1$, then the same arguments as for \cite[Corollary A]{Ber19} can be repeated to show that $(f, u)$ is a pair of endomorphisms of a formal group defined over $\mathcal{O}_K$.

    \subsection*{Structure of the article} In the first section, we begin by recalling some invariants and relations attendant to a $p$-adic dynamical system $(f, u)$ from \cite{Lub94}. Using the main theorem of \cite{Li96}, we determine the decreasing part of the Newton polygon of any iterate of $f$. This yields information on the roots of the iterates of $f$ and also on the periodic points of $u$. The situation slightly differs from the case of \cite{Spe18} but a similar path can be followed modulo some new adjustments and arguments. The second section is devoted to the recovery of the $p$-adic Tate module structure on the set $T_f$ attached to $(f, u)$. In \S\ref{S: dynamical-recovery-GL}, we mimic as in \cite{Spe18} the Galois action on $T_f$ by means of the action of the $\mathbb{Z}_p$-iterates of $u$ in order to attach a character $\chi_f \colon G_L \rightarrow \mathbb{Z}_p^\times$ to $(f, u)$, where $L$ is the finite extension obtained by adjoining to $K$ the fixed points of $u$. In \S\ref{S: Fontaine-period-rings} and \S\ref{S: universal-consistent-ring}, we expose the material used in \S\ref{S: regularity-GL} in order to construct a crystalline period of $\chi_f$. We then prove that it has weight $1$. These properties carry the existence of a formal group $H$ over $\mathcal{O}_L$ whose Tate module is isomorphic to $\mathbb{Z}_p(\chi_f)$. Although $H$ need not be latent to the dynamical system, it should be isomorphic to it. Thus we rely on $H$ as a cornerstone of our constructions. In \S\ref{S: define-Tate-module}, we begin by identifying the counterpart of $(f, u)$ in $H$. This allows us, in a second time, to transport the structure of $\T_p(H)$ to $T_f$ and to turn the latter into a $\mathbb{Z}_p[G_L]$-module for which $f$ is an endomorphism. In the third section, we move backward in explicit integral $p$-adic Hodge theory until we get a formal group.~From \cite{Kis06} we consider the associated Breuil-Kisin module in \S\ref{S: latent-Breuil-Kisin-module}, then extend it into a window as in \cite{Kis09} in \S\ref{S: latent-window}, and transform it into a connected $p$-divisible group as in \cite{Zin01} in \S\ref{S: latent-p-divisible-group}. At every stage, we track the behavior of $f$ on the new object, allowing us to conclude in \S\ref{S: latent-formal-group} that the formal group associated to the connected $p$-divisible group is the latent one.

    {\small\subsection*{Acknowledgments} I am sincerely grateful to my supervisor Maja Volkov for her guidance and invaluable advice and to Giovanni Bosco for helpful conversations on $p$-adic Hodge theory. I am also particularly grateful to Laurent Berger for many valuable comments on an earlier version of this paper and to Jonathan Lubin for his kind support.}

    \setcounter{equation}{1}

\section{The dynamical system}

\noindent Let $K/\mathbb{Q}_p$ be a finite extension satisfying $(e_K, p(p-1)) = 1$. Throughout this article, we assume $(f, u)$ to be a commuting pair of formal power series in $\mathscr{S}_0(\mathcal{O}_K)$ such that $f \bmod \mathfrak{m}_K \neq 0$, the roots of $f$ and all of its iterates in $\mathfrak{m}_{\mathbb{C}_p\!}$ are simple and $f$ has $p$ of them, where $f'(0)$ is a uniformizer in $\mathbb{Z}_p$ and $u'(0) \in \mathbb{Z}_p^\times$ is nontorsion. Up to replacing $u$ by one of its $\mathbb{Z}$-iterates, we may assume that $u'(0) \in 1 + p\mathbb{Z}_p$. In this section, we review the immediate data at hand in such a dynamical system which will play a crucial role in the subsequent work.

    \subsection{Logarithm}\label{S: logarithm} In his seminal article \cite{Lub94}, Lubin showed that, in the presence of a formal power series $s \in \mathscr{S}_0(\mathcal{O}_K)$ such that $s'(0)$ is nonzero nor a root of unity, there exists a unique invertible power series
    \[
        \Log_s \in \mathscr{S}_0(K)
    \]
    satisfying $\Log_s'(0) = 1$ and $\Log_s(s(X)) = s'(0)\Log_s(X)$, called the logarithm of $s$. Every power series satisfying the same condition as above and commuting with $s$ has the same logarithm. Denote by $\Log_f$ the logarithm of the commuting pair~$(f, u)$; it converges on $\mathfrak{m}_{\mathbb{C}_p\!}$. If $f$, $u$ are endomorphisms of an integral formal group, then $\Log_f$ is the logarithm of this formal group and it is well known that its derivative has coefficients in $\mathcal{O}_K$. It turns out that this is true even for our $p$-adic dynamical system $(f, u)$ by \cite[Lemma 2.2]{Ber19}. The conjugated
    \[
        F(X,Y) = \Log_f^{\circ -1}(\Log_f(X) + \Log_f(Y))
    \]
    is the unique formal group for which $f$ and $u$ are endomorphisms, but it is defined {\em a priori} over $K$ rather than $\mathcal{O}_K$. Therefore solving Lubin's conjecture amounts to show that $F$ is defined over $\mathcal{O}_K$. For every $m \in \mathbb{Z}_p$, we let
    \[
        [m]_f(X) = \Log_f^{\circ -1}(m\Log_f(X)) \in \mathscr{S}_0(K)
    \]
    be the unique series over $K$ commuting with $f$ and acting on the tangent space at $0$ by $m$, which is nothing but the multiplication-by-$m$ endomorphism in $F$.

    \subsection{Torsion points} Whenever we talk about the roots or the fixed points of a power series, we mean its roots or its fixed points, respectively, in the $p$-adic open unit disk $\mathfrak{m}_{\mathbb{C}_p\!}$. A noninvertible power series $s_1 \in \mathscr{S}_0(\mathcal{O}_K)$ such that $s_1'(0)$ is nonzero can have no other fixed points than $0$, but many roots. The set of roots of $s_1$ and its iterates is denoted by
    \[
        \Lambda(s_1) = \{x \in \mathfrak{m}_{\mathbb{C}_p}\! \mid \exists\, n \geqslant 0,\; s_1^{\circ n}(x) = 0\}.
    \]
    On the other hand, a nontorsion invertible power series $s_2 \in \mathscr{S}_0(\mathcal{O}_K)$ can have no other roots than $0$, but many fixed points. If $s_1$ and $s_2$ commute with each other, Lubin showed in \cite[Proposition 3.2]{Lub94} that $\Lambda(s_1)$ is the set of periodic points of $s_2$. Knowing more about these points (their valuation, how many there are for a given valuation) is a key point in our arguments. This can be achieved by means of the Weierstrass preparation theorem and Newton polygons. Working with Newton polygons requires fixing a valuation; for this, we always use $\upsilon_K$ the one normalized on $K$ so that $\upsilon_K(p) = e$.

    \begin{proposition}\label{R: roots-iterates-f}
        For every integer $n \geqslant 1$, the endpoints of the decreasing part of the Newton polygon of $f^{\circ n}\!$ are the $(p^i, e(n-i))$ for each integer $i \in \{0, \dots, n\}$.
    \end{proposition}

    \begin{proof}
        We proceed by induction on $n$. The base case ($n = 1$) follows directly from the hypotheses on $f$ and \cite[Theorem 4.1]{Li96}. The decreasing part of the Newton polygon of $f$ consists of a single line of slope $-e/(p-1)$, and this implies that the Weierstrass polynomial associated to $f(X)/X$ is irreducible over $K$ given that we assume $(e, p-1) = 1$. Assume now that this holds for $n \geqslant 1$. By multiplicativity of the linear coefficients and the Weierstrass degrees,
        \[
            (1, e(n+1)) \quad\text{and}\quad (p^{n+1}, 0)
        \]
        are vertices of the Newton polygon of $f^{\circ n+1}$. By varying $\pi$ among the $p$ roots of $f$, the equations $f^{\circ n} - \pi = 0$ describe all the roots of $f^{\circ n+1}$. When $\pi = 0$, one then retrieves the roots of $f^{\circ n}$. Otherwise, the Newton polygon of $f^{\circ n} - \pi$ consists of a single line of slope $-e/(p^{n+1}-p^n)$, which is greater than all those of the Newton polygon of $f^{\circ n}$ and therefore contributes to the last segment. Packing the roots of $f^{\circ n+1}/f^{\circ n}$, that is, the roots of $f^{\circ n+1}$ that are not roots of $f^{\circ n}$, one thus finds that there are $p^n(p-1)$ of them and that the Weierstrass polynomial associated to this segment is irreducible over $K$ since we assume $(e, p(p-1)) = 1$. The left-hand side part of the Newton polygon of $f^{\circ n+1}$ comes from the induction hypothesis.
    \end{proof}

    \begin{remark}\label{O: irreducible-polynomials}
        This proof yields two observations:
        \begin{enumerate}
            \item Unlike Specter's case, the Weierstrass polynomial of $f^{\circ n+1}/f^{\circ n}$ is not Eisenstein in general but still irreducible over $K$, which is sufficient: therefore the absolute Galois group $G_K$ of $K$ acts transitively on the roots of $f^{\circ n+1}/f^{\circ n}$ since $f$ is defined over $\mathcal{O}_K$.
            \item Asking for the simplicity of the roots of $f$ and its iterates is equivalent to asking that the $p-1$ nonzero roots of $f$ have the same valuation, under the coprimality assumption. This reformulation is easier to check in practice because it only requires to draw a finite part of the Newton polygon of $f$.
        \end{enumerate}
    \end{remark}

    \begin{corollary}\label{R: periodic-points-u}
        For every nonzero $m \in \mathbb{Z}_p$, the fixed points of $u^{\circ m}\!$ are the roots of $f^{\circ v(m)}$ where $v(m) = \upsilon_K(u'(0)^m - 1)/e$.
    \end{corollary}

    \begin{proof}
        The series $u$ commutes with the noninvertible $f$ having finite Weierstrass degree, ensuring that its nontrivial iterates have finitely many fixed points by \cite[Corollary 4.3.1]{Lub94}. The decreasing part of the Newton polygon of $u^{\circ m}(X) - X$ evolves between the vertex $(1, \upsilon_K(u'(0)^m - 1))$ and the $X$-axis. Since $u$ is defined over $\mathcal{O}_K$, its fixed points are a finite union of $G_K$-orbits in $\Lambda(f)$. According to Remark \ref{O: irreducible-polynomials}, each segment in the latter polygon then makes a decrease of $e$ along the $Y$-axis.
    \end{proof}

\section{The latent Tate module}

\noindent As stated in the introduction, the $p$-adic Tate module of a formal group characterizes it up to isomorphism. Specter provided a method in \cite{Spe18} to recover the associated representation. In this section, we adapt this method to our case and go further by retrieving the full structure on the set of $f$-consistent sequences.

    \subsection{The Galois action on \texorpdfstring{$f$}{f}-consistent sequences}\label{S: dynamical-recovery-GL} The full structure on $T_f$ can be recovered solely from the Galois action. To this end, we are interested in finding a more explicit description of this action. This can be achieved by comparing it with the action of the $\mathbb{Z}_p$-iterates of $u$ on $T_f$. We recall that $T_f$ is the $G_K$-set of $f$-consistent sequences of elements in $\mathfrak{m}_{\mathbb{C}_p\!}$, and that the pair of series $(f, u)$ is a pair of endomorphisms of it, acting by componentwise evaluation:
    \[
        T_f = \big\{(x_n)_{n \geqslant 0} \in \mathfrak{m}_{\mathbb{C}_p}^\mathbb{N}\! \;\big|\; x_0 = 0 \text{ and } \forall\, n \geqslant 0,\; f(x_{n+1}) = x_n \big\}.
    \]
    Let $S_f \subset T_f$ be the subset composed of all elements whose second entry is nonzero. If Lubin's conjecture holds, $T_f$ is canonically isomorphic to the $p$-adic Tate module of the latent formal group, and an element $x \in T_f$ is a generator of $T_f$ as $\mathbb{Z}_p$-module if and only if $x \in S_f$. Elements of $T_f$ are either zero or have finitely many initial zero entries followed by an element of $S_f$. In order to describe the $G_K$-action on $T_f$, it is therefore sufficient to describe it on the subset $S_f$. Nevertheless, recovering this action dynamically via the $\mathbb{Z}_p$-iterates of $u$ requires restricting it. Indeed, $u$, and hence all its iterates, fix some initial entries of all the elements of $T_f$. We let
    \[
        \ell = \upsilon_K(u'(0) - 1)/e \geqslant 1
    \]
    which is a positive integer since we have picked $u'(0) \in 1 + p\mathbb{Z}_p$. By Corollary \ref{R: periodic-points-u}, $u$ fixes up to the $\ell+1$ first components of the elements of $S_f$. We interpret $S_f$ as a directed tree, as shown in Figure \ref{F: tree}. Indexed from bottom to top, its vertices at level $n \geqslant 1$ are the roots of $f^{\circ n}/f^{\circ n-1}$, and there is an edge from a vertex $v_2$ to a vertex $v_1$ if and only if $f(v_2) = v_1$. Given $\pi = (\pi_n)_{n \geqslant 0} \in S_f$, we let
    \[
        S_{f, \pi} = \{(x_n)_{n \geqslant 0} \in T_f \mid x_0 = \pi_0,\dots, x_\ell = \pi_\ell\}
    \]
    which is a subset of $S_f$. Set $j = p^\ell - p^{\ell-1}$; any sequence $(\alpha_1,\dots,\alpha_j)$ of elements $\alpha_i = (\alpha_{i,n})_{n \geqslant 0}$ of $S_f$ such that $\{\alpha_{1,\ell},\dots,\alpha_{j,\ell}\}$ is the set of roots of $f^{\circ \ell}/f^{\circ \ell-1}$ then determines a partition $(S_{f, \alpha_1},\dots, S_{f, \alpha_j})$ of $S_f$. In what follows, we first retrieve the Galois action on such subsets and then discuss of its global action on $S_f$. The next result allows us to pass from one element of $S_f$ to another.

    \begin{figure}
        \begin{center}
            \begin{tikzcd}[column sep=0mm, row sep=5mm]
                & \;\vdots\ar[dr] & \vdots\ar[d] & \vdots\;\ar[dl] & & \;\vdots\ar[dr] & \vdots\ar[d] & \vdots\;\ar[dl] & & \;\vdots\ar[dr] & \vdots\ar[d] & \vdots\;\ar[dl] & & & & \;\vdots\ar[dr] & \vdots\ar[d] & \vdots\;\ar[dl] & & \;\vdots\ar[dr] & \vdots\ar[d] & \vdots\;\ar[dl] & & \;\vdots\ar[dr] & \vdots\ar[d] & \vdots\;\ar[dl]\\
                \ell \ar[rr, dash, densely dotted] & & \ast \ar[drrrr]\ar[rrrr, dash, densely dotted] & & & & \ast \ar[d]\ar[rrrr, dash, densely dotted] & & & & \ast \ar[dllll]\ar[rrrrrr, dash, densely dotted] & & & & & & \ast \ar[drrrr]\ar[rrrr, dash, densely dotted] & & & & \ast \ar[d]\ar[rrrr, dash, densely dotted] & & & & \ast \ar[dllll]\ar[rr, dash, densely dotted] & & \phantom{\ell}\\
                & & & & & & \ast \ar[drrrrrrr] & & & & & & & & & & & & & & \ast \ar[dlllllll]\\
                & & & & & & & & & & & & & 0
            \end{tikzcd}
        \end{center}
        \caption{Directed tree associated to $S_f$ in the case $p = 3$ and $\ell = 2$.}
        \label{F: tree}
    \end{figure}

    \begin{proposition}\label{R: transitive-action-GK-Sf}
        The Galois group $G_K\!$ acts transitively on the set $S_f$.
    \end{proposition}

    \begin{proof}
        For each $n \geqslant 1$, the $n$-th component of an element of $S_f$ is a root of $f^{\circ n}$ that is not a root of $f^{\circ n-1}$. These roots are the roots of an irreducible polynomial over $K$, as stated in Remark \ref{O: irreducible-polynomials}. Therefore $G_K$ acts transitively on this set. To prove that the transitivity passes to the limit, pick $(x_n)_{n \geqslant 0}$, $(y_n)_{n \geqslant 0} \in S_f$ and define for all $n \geqslant 0$ the set
        \[
            A_n = \{\sigma \in G_K \mid \sigma.x_n = y_n\}
        \]
        which is closed. Since $G_K$ is compact, the nested sequence $(A_n)_{n \geqslant 0}$ has a nontrivial intersection. This proves the proposition.
    \end{proof}

        \subsubsection{Action on cells} Let $\pi = (\pi_n)_{n \geqslant 0}$ be an element of $S_f$ and set $L_\pi = K(\pi_\ell)$. Then the group $G_{L_\pi}\!$ fixes the $\ell+1$ first components of $\pi$. We begin by recovering the action of $G_{L_\pi}\!$ on the cell $S_{f, \pi} \subset S_f$.

        \begin{proposition}\label{R: simple-transitive-action-u-Sfpi}
            The group $u^{\circ \mathbb{Z}_p}\!$ acts simply transitively on the set $S_{f, \pi}$.
        \end{proposition}

        \begin{proof}
            This is the same as \cite[Proposition 2.9]{Spe18} adapting the notations. For the reader's convenience and due to the stature of this result, we outline the proof. To show that the action is transitive, it suffices to show, by compacity of $\mathbb{Z}_p$, that it is the case on
            \[
                O_{n}(\pi_\ell) = \{x \in \mathfrak{m}_{\mathbb{C}_p}\! \mid f^{\circ n-\ell}(x) = \pi_\ell\} \subset \Lambda(f)
            \]
            for every integer $n > \ell$. Using the Newton polygon, one finds that the cardinality of $O_n(\pi_\ell)$ is $p^{n-\ell}$. Joint to Corollary \ref{R: periodic-points-u}, the stabilizer of an element $x \in O_n(\pi_\ell)$ under the action of the $\mathbb{Z}_p$-iterates of $u$ is
            \begin{align*}
                \Stab(x) & = \big\{u^{\circ m} \in u^{\circ \mathbb{Z}_p} \;\big|\; \upsilon_K(u'(0)^m - 1) \geqslant ne\big\}\\
                & = \big\{u^{\circ m} \in u^{\circ \mathbb{Z}_p} \;\big|\; \upsilon_K(u'(0) - 1) + \upsilon_K(m) \geqslant ne\big\}\\
                & = u^{\circ p^{n-\ell} \mathbb{Z}_p}
            \end{align*}
            since we have picked $u'(0) \in 1 + p\mathbb{Z}_p$. By the orbit-stabilizer theorem, we deduce that $u^{\circ \mathbb{Z}_p}$ acts transitively on $O_n(\pi_\ell)$. Therefore $u^{\circ \mathbb{Z}_p}$ acts transitively on $S_{f, \pi}$, and this action is free because of \cite[Corollary 4.3.1]{Lub94}.
        \end{proof}

        The transitive action of $G_K$ on the set $S_f$ restricts to a transitive action of $G_{L_\pi\!}$ on the subset $S_{f, \pi}$. Combining this observation with Proposition \ref{R: simple-transitive-action-u-Sfpi} enables us to recover the latter action via the following rule: as $g$ varies in $G_{L_\pi}$, the elements $g.\pi$ exhaust the set $S_{f, \pi}$, and comparing $g.\pi$ with $\pi$ gives a unique $\mathbb{Z}_p$-iterate $u_g$ of $u$ such that
        \[
            g.\pi = u_g(\pi).
        \]
        The resulting map sending $g$ to $u_g'(0)$ defines a character.

        \begin{proposition}
            The map $\chi_{f, \pi} \colon G_{L_\pi\!} \rightarrow \mathbb{Z}_p^\times$ defined by $g \mapsto u_g'(0)$ is a character of $G_{L_\pi}\!$ satisfying $g.\pi = [\chi_{f, \pi}(g)]_f(\pi)$ for every $g \in G_{L_\pi\!}$.
        \end{proposition}

        \begin{proof}
            The fact that $\chi_{f, \pi}$ is a group homomorphism is a consequence of the freeness of the action in Proposition \ref{R: simple-transitive-action-u-Sfpi}. The rest of the statement is immediate.
        \end{proof}

        \subsubsection{Action on the whole set} Having analyzed the situation on cells, we turn to the global case. To compare the local actions with each other, we consider the finite extension $L$ of $K$ generated by the roots of $f^{\circ \ell}/f^{\circ \ell-1}$ or, equivalently, the Galois closure of $L_\pi/K$ for any $\pi \in S_f$. Given $\pi, \pi' \in S_f$, Proposition \ref{R: transitive-action-GK-Sf} guarantees the existence of a $\sigma \in G_K$ satisfying $\pi' = \sigma.\pi$. Therefore $L_{\pi'} = \sigma(L_\pi)$ and the absolute Galois groups of $L_\pi$ and $L_{\pi'}$ are $\sigma$-conjugate in $G_K$. Since $u$ is defined over $\mathcal{O}_K$, the previous rule gives, for every $g \in G_L$, that
        \begin{equation}\label{fact-2}
            \chi_{f, \pi}(g) = \chi_{f, \pi'}(\sigma g \sigma^{-1}).
        \end{equation}
        Now if Lubin's conjecture holds, we should also obtain equality of the characters, since there would exist a unique character giving the action of $G_K$ at every point of $T_f$. This lack of precision in our situation reflects the price we pay for restricting the Galois action from $K$ to $L$. Nevertheless, we fix once and for all a $\pi = (\pi_n)_{n \geqslant 0} \in S_f$ and let
        \[
            \chi_f \colon G_L \longrightarrow \mathbb{Z}_p^\times \colon g \longmapsto \chi_{f, \pi}(g).
        \]
        After the two preparatory subsections that follow, we shall prove that this character is crystalline of weight $1$. Consequently there exists a formal group defined over $\mathcal{O}_L$ whose associated character is $\chi_f$, and we transport the structure of its $p$-adic Tate module to $T_f$. Theorem \ref{R: main-theorem} will ultimately imply that the subsequent constructions were independent of the choice of $\pi$.

    \subsection{Fontaine's period rings}\label{S: Fontaine-period-rings} Fontaine has classified the $p$-adic representations of the absolute Galois group $G_E$ of a finite extension $E/\mathbb{Q}_p$. His classifier relies on period rings such as $\BdR$ and $\Bcris$. In this subsection, we outline the construction of these two rings and review the formal group setting. The original constructions can be found in \cite{Fon94}, and in \cite{Col02} for the tensorial generalization. To coincide with our situation, we work directly in the particular case of $E = L$.

        \subsubsection{The integers of the tilt of \texorpdfstring{$\mathbb{C}_p$}{Cp}} Let $\mathcal{O}_{\mathbb{C}_p}\!$ be the closed unit disk of $\mathbb{C}_p$. Throughout the remainder of this article, we fix $\pi_L$ a uniformizer of $L$ and we let $q$ be the cardinality of the residue field $\mathcal{O}_L/\mathfrak{m}_L = \mathbb{F}_q$. The definition of the aforementioned period rings involves several mechanisms, whose starting point is the perfect ring
        \[
            \mathcal{R} = \varprojlim_{x \mapsto x^q} \mathcal{O}_{\mathbb{C}_p}/(\pi_L)
        \]
        of characteristic $p$. This ring is endowed with a valuation $\upsilon_\mathcal{R}$ for which it is complete, defined by $\upsilon_\mathcal{R}(x) = \lim_{n \to +\infty} q^n\upsilon_p(\widehat{x}_n)$ for all $x = (x_n)_{n \geqslant 0} \in \mathcal{R}$, where $\widehat{x}_n \in \mathcal{O}_{\mathbb{C}_p\!}$ is a lift of $x_n$ modulo $(\pi_L)$ for all integers $n \geqslant 0$ and where $\upsilon_p$ is the $p$-adic valuation normalized on $\mathbb{Q}_p$. Moreover, this ring is equipped with a compatible action of $G_L$ inherited from the action of the latter on $\mathcal{O}_{\mathbb{C}_p\!}$.

        \subsubsection{Blank period rings} We move to characteristic zero. Let $\W_{\!\pi_L\!}$ be the generalized $\pi_L$-adic Witt vectors endofunctor of $\mathcal{O}_L$-algebras; see for instance \cite[\S2]{CD15} for a short presentation in our context. The ring $\mathcal{R}$ is naturally an $\mathcal{O}_L$-algebra and we let
        \[
            \AinfL = \W_{\!\pi_L\!}(\mathcal{R})
        \]
        be the unique strict $\pi_L$-ring of residue ring $\mathcal{R}$. Alternatively, one can see this blank period ring as $\mathcal{O}_L \otimes_{\W(\mathbb{F}_q)} \W(\mathcal{R})$, where $\W$ is the usual $p$-adic Witt vectors functor. Its elements can be uniquely written in the form $\sum_{n \geqslant 0} [x_n]\pi_L^n$, where $x_n \in \mathcal{R}$ and $[-]$ is the Teichmüller lift. The action of $G_L$ on $\mathcal{R}$ lifts by functoriality to an action on $\AinfL$. An important map attached to this ring is the $G_L$-equivariant epimorphism $\theta \colon \AinfL \twoheadrightarrow \mathcal{O}_{\mathbb{C}_p}\!$ of $\mathcal{O}_L$-algebras determined by mapping $[x]$ to $\lim_{n \to +\infty} \widehat{x}_n^{q^n}\!$, where the element $\widehat{x}_n \in \mathcal{O}_{\mathbb{C}_p}\!$ lifts the $n$-th component of $x \in \mathcal{R}$. This map uniquely extends by $L$-linearity to a $G_L$-equivariant epimorphism $\theta$ from $\BinfL^+ = \AinfL[1/\pi_L]$ to $\mathbb{C}_p$ and its kernel is principal.

        \subsubsection{The de Rham period field} The de Rham period ring
        \[
            \BdR^+ = \varprojlim_{n \geqslant 1} \BinfL^+/\ker(\theta)^n
        \]
        is the $\ker(\theta)$-adic completion of $\BinfL^+$. It inherits the action of $G_L$ on $\BinfL^+$, which is compatible with the $\ker(\theta)$-adic topology. Again, $\theta$ extends to a $G_L$-equivariant epimorphism $\theta \colon \BdR^+ \twoheadrightarrow \mathbb{C}_p$ and its kernel remains principal. The $\mathbb{Z}$-powers of the ideal $\ker(\theta)$ induce a decreasing, exhaustive and separated filtration $(\Fil^n \BdR)_{n \in \mathbb{Z}}$ on the fraction field $\BdR$ of $\BdR^+$. This field is our first ring of interest.

        The field $\BdR$ is sufficient to capture $p$-adic representations arising from formal groups, but too coarse to be used to classify them. Let $\REP_{\mathbb{Q}_p}(G_L)$ be the category of $p$-adic representations of $G_L$, and let $\FIL_L$ be the category of finite-dimensional $L$-vector spaces equipped with a decreasing, exhaustive and separated $\mathbb{Z}$-filtration by subspaces. Fontaine defined a contravariant functor
        \[
            \D_\mathrm{dR}^* \colon \REP_{\mathbb{Q}_p}(G_L) \longrightarrow \FIL_L
        \]
        by sending a $p$-adic representation $V$ to $\D_\mathrm{dR}^*(V) = \Hom_{\mathbb{Q}_p[G_L]}(V, \BdR)$, equipped with the induced filtration $\Fil^n\D_\mathrm{dR}^*(V) = \Hom_{\mathbb{Q}_p[G_L]}(V, \Fil^n\BdR)$ for all $n \in \mathbb{Z}$.~We say that a $p$-adic representation $V$ of $G_L$ is de Rham if the $L$-dimension of $\D_\mathrm{dR}^*(V)$ is equal to the $\mathbb{Q}_p$-dimension of $V$. In this case, the jumps of the filtration are the Hodge-Tate weights of $V$. Restricting the latter functor to de Rham representations yields a faithful functor. When $(V, \rho)$ is one-dimensional, being de Rham is equivalent to the existence of a nonzero period $z \in \BdR$ such that $g.z = \rho(g)z$ for all $g \in G_L$. Then its weight is the largest $w \in \mathbb{Z}$ such that $z \in \Fil^w \BdR$.

        \subsubsection{The crystalline period ring} To refine the theory, Fontaine defined a subring $\Bcris$ of $\BdR$ as follows: let $\Acris^\circ$ be the PD-envelope of $\W(\mathcal{R})$ with respect to $\ker(\theta)$. Set $\Acris$ to be its $p$-adic completion; one can show that it is still a subring of $\BdR$. Define $\Bcris^+ = \Acris[1/p]$. The $G_L$-action on $\W(\mathcal{R})$ extends at each previous ring and the latter contains a unique $G_L$-stable $\mathbb{Q}_p$-line on which the group $G_L$ acts by the cyclotomic character. The crystalline period ring
        \[
            \Bcris = \Bcris^+[1/t]
        \]
        is obtained from $\Bcris^+$ by inverting any nonzero $t$ in this line. Unlike in the case of $\BdR$, the Frobenius $\varphi$ on $\W(\mathcal{R})$ extends uniquely to $\Bcris$, where $\varphi(t) = pt$, into an injective $G_L$-equivariant endomorphism $\varphi \colon \Bcris \hookrightarrow \Bcris$ which is semilinear with respect to the Frobenius $\sigma$ of $L_0$, the maximal unramified extension of $\mathbb{Q}_p$ inside $L$. The canonical map $L \otimes_{L_0} \Bcris \rightarrow \BdR$ is injective and $G_L$-equivariant.

        The ring $\Bcris$ will be of primary interest here, as it is rich enough not only to capture but also to classify representations arising from formal groups. Let $\MF_L(\varphi)$ be the category of finite-dimensional $L_0$-vector spaces with a bijective $\sigma$-semilinear Frobenius endomorphism $\varphi$, and whose scalar extension to $L$ are objects of $\FIL_L$. Morphisms in this category are $L_0$-linear maps commuting with $\varphi$ and preserving the filtration after scalar extension. Fontaine defined a contravariant functor
        \[
            \D_\mathrm{cris}^* \colon \REP_{\mathbb{Q}_p}(G_L) \longrightarrow \MF_L(\varphi)
        \]
        by sending a $p$-adic representation $V$ to $\D_\mathrm{cris}^*(V) = \Hom_{\mathbb{Q}_p[G_L]}(V, \Bcris)$, equipped with the Frobenius induced by that on $\Bcris$. One has $L \otimes_{L_0} \D_\mathrm{cris}^*(V) \hookrightarrow \D_\mathrm{dR}^*(V)$. We say that a $p$-adic representation $V$ of $G_L$ is crystalline if the $L_0$-dimension of $\D_\mathrm{cris}^*(V)$ is equal to the $\mathbb{Q}_p$-dimension of $V$. The Hodge-Tate weights of a crystalline representation are its weights as a de Rham representation. Restricting the previous functor to crystalline representations yields an anti-equivalence of categories
        \[
            \D_\mathrm{cris}^* \colon \REP_{\mathbb{Q}_p}^\mathrm{cris}(G_L) \longrightarrow \MF_L^\mathrm{adm}(\varphi)
        \]
        where the superscript “cris'' stands for the crystalline ones and “adm'' for the~admissible ones. Concretely, a one-dimensional filtered $\varphi$-module $D$ with jump at $j \in \mathbb{Z}$ is admissible if and only if $j = \upsilon_p(\varphi(d)/d)$ where $d \in D$ is any nonzero vector.

        \subsubsection{Application to formal groups}\label{S: formal-groups-situation} Given a formal group $H$ of height $h \geqslant 1$ defined  over $\mathcal{O}_L$, the $\mathbb{Q}_p$-scalar extension of its $p$-adic Tate module
        \[
            \V_{\!p}(H) = \mathbb{Q}_p \otimes_{\mathbb{Z}_p} \T_p(H)
        \]
        is a de Rham $G_L$-representation of dimension $h$ and of Hodge-Tate weights $0$ with multiplicity $h-1$ and $1$ with multiplicity $1$. Moreover, $\V_{\!p}(H)$ is a crystalline representation of $G_L$. In this paper, we are primarily interested in the integral setting for which the situation is as follows. Denote by $\FGL/\mathcal{O}_L$ the category of formal groups of finite height and finite dimension defined over $\mathcal{O}_L$ and by $\REP_{\mathbb{Z}_p}^{\mathrm{cris},\mathrm{fgl}}(G_L)$ the category of $G_L$-stable $\mathbb{Z}_p$-lattices in crystalline $G_L$-representations arising from the objects in $\FGL/\mathcal{O}_L$. According to \cite[Corollary 6.2.3]{SW13} and \cite[Proposition 1]{Tat67}, the functor
        \[
            \T_p \colon \FGL/\mathcal{O}_L \longrightarrow \REP_{\mathbb{Z}_p}^{\mathrm{cris},\mathrm{fgl}}(G_L)
        \]
        is an equivalence. When $H$ is of height $1$, any endomorphism $s \in \End_{\mathcal{O}_L}(H)$ acts on $\T_p(H)$ by multiplication by its linear coefficient $s'(0) \in \mathbb{Z}_p$. In view of this equivalence, the single coefficient $s'(0)$ together with the $\mathbb{Z}_p$-module structure on $\T_p(H)$ or $\V_{\!p}(H)$ uniquely determine the power series $s$. It also acts on $\D_\mathrm{dR}^*(\V_{\!p}(H))$ in the following way: for a nonzero $v \in \V_{\!p}(H)$ and for a nonzero period $z \in \BdR$ of $\V_{\!p}(H)$, the map induced by $s$ sends $(v \mapsto z)$ to $(v \mapsto s'(0)z)$ since $s(v) = s'(0)v$.  The initial series $s$ is unequivocally determined by this action.

    \subsection{The universal \texorpdfstring{$f^{\circ r}$}{fr}-consistent ring}\label{S: universal-consistent-ring} Denote by $r$ the residue degree of $L/\mathbb{Q}_p$, so that $q = p^r$. To suit this parameter, we adapt the universal $f$-consistent ring in \cite{Spe18} by replacing $f$ with its $r$-th iterate. According to \cite[Corollary 6.2.1]{Lub94}, there exists an invertible power series $a \in \mathscr{S}_0(\mathbb{F}_q)$ satisfying
    \[
        f^{\circ r}(X) \bmod (\pi_L) = a(X^q).
    \]
    We consider the colimit
    \[
        \Aty^\circ(\mathcal{O}_L) = \varinjlim_{n \geqslant 0} \mathcal{O}_L[\![X_n]\!]
    \]
    with the identification $X_n \mapsto f^{\circ r}(X_{n+1})$ for all integers $n \geqslant 0$. We let $\Aty(\mathcal{O}_L)$ be its $\pi_L$-adic completion and equip it with a $G_L$-action given by $g.X_n = [\chi_f(g)]_f(X_n)$. Recall that by construction $[\chi_f(g)]_f \in \mathscr{S}_0(\mathcal{O}_K)$ since it is a $\mathbb{Z}_p$-iterate of $u$.

    \begin{proposition}\label{R: UCR-strict}
        The ring $\Aty(\mathcal{O}_L)$ is a strict $\pi_L$-ring.
    \end{proposition}

    \begin{proof}
        The ring $\Aty(\mathcal{O}_L)$ is a strict $\pi_L$-ring if it is $\pi_L$-adically complete, $\pi_L$-torsion free and $\Aty(\mathcal{O}_L)/(\pi_L)$ is a perfect ring. The first two conditions are satisfied by definition. To show the third one, observe that
        \begin{align*}
            \Aty(\mathcal{O}_L)/(\pi_L)
            & \simeq \Aty^\circ(\mathcal{O}_L)/(\pi_L)\\
            & \simeq \varinjlim_{X_n \mapsto a(X^q_{n+1})} \mathcal{O}_L[\![X_n]\!]/(\pi_L)\\
            & \simeq \varinjlim_{X_n \mapsto a(X_{n+1})^q} \mathbb{F}_q[\![X_n]\!]\\
            & \simeq \varinjlim_{a^{\circ n}(X_n) \mapsto a^{\circ n+1}(X_{n+1})^q} \mathbb{F}_q[\![a^{\circ n}(X_n)]\!]. \qedhere
        \end{align*}
    \end{proof}

    \begin{up}
        Endowed with the $(\pi_L, X_0)$-adic topology, every formal~variable $X_n$ is topologically nilpotent in $\Aty(\mathcal{O}_L)$. Furthermore, this ring is universal for $f^{\circ r}$-consistent sequences in the following sense: given a pair $(R, s)$, where $R$ is a $\pi_L$-adically complete, ind-complete adic $\mathcal{O}_L$-algebra, $s = (s_n)_{n \geqslant 0}$ is a $f^{\circ r}$-consistent sequence of topologically nilpotent elements in $R$, there is a unique homomorphism $\phi_s \colon \Aty(\mathcal{O}_L) \rightarrow R$ sending $X_n$ to $s_n$ for all $n \geqslant 0$. Conversely, any homomorphism from $\Aty(\mathcal{O}_L)$ to $R$ arises in this way for some sequence $s$.
    \end{up}

    \subsection{Regularity of the Galois action}\label{S: regularity-GL} We prove that the character $\chi_f = \chi_{f, \pi}$ chosen in \S\ref{S: dynamical-recovery-GL} is crystalline of Hodge-Tate weight $1$. For this, we pass to the ring $\mathcal{R}$ via the map
    \[
        \vartheta_f \colon T_f \longrightarrow \mathcal{R} \colon (y_n)_{n \geqslant 0} \longmapsto (a^{\circ n}(y_{nr} \bmod \pi_L))_{n \geqslant 0}.
    \]
    Then we apply the universal property of the ring $\Aty(\mathcal{O}_L)$ to the pair $(\mathcal{R}, \varpi)$, where $\varpi = (\varpi_n)_{n \geqslant 0}$ is the $f^{\circ r}$-consistent sequence of topologically nilpotent elements in $\mathcal{R}$ given by
    \[
        \varpi_n = (a^{\circ m}(\pi_{(n+m)r} \bmod \pi_L))_{m \geqslant 0}
    \]
    for all $n \geqslant 0$. Let $\phi_\varpi \colon \Aty(\mathcal{O}_L) \rightarrow \mathcal{R}$ be the corresponding homomorphism. As $\pi_L$ is zero in $\mathcal{R}$, this map factors through $(\pi_L)$; it is even its kernel.

    \begin{proposition}
        The previous homomorphism $\phi_\varpi$ is $G_L$-equivariant with kernel $\pi_L\Aty(\mathcal{O}_L)$.
    \end{proposition}

    \begin{proof}
        The first assertion follows from the fact that all the power series $[\chi_f(-)]_f$ commute with $f$. The second is the same as in \cite[Lemma 2.2]{Ber16}: if the induced map $\mathbb{F}_q[\![X_n]\!] \rightarrow \mathcal{R}$ were not injective, there would exist a nonzero $P(X_n) \in \mathbb{F}_q[X_n]$ such that $P(\overline{\phi_\varpi}(X_n)) = 0$ and this would imply that $\varpi_n$ is algebraic over $\mathbb{F}_p$, which is not the case.
    \end{proof}

    Proposition \ref{R: UCR-strict} asserts that $\Aty(\mathcal{O}_L)$ is a strict $\pi_L$-ring, consequently the factorization of $\phi_\varpi$ through its kernel lifts to an injective $G_L$-equivariant homomorphism $\Aty(\mathcal{O}_L) \hookrightarrow \AinfL$. Under this embedding, we let $x_n$ be the image of $X_n$ for every integer $n \geqslant 0$. The action of the group $G_L$ on these points is particularly interesting, as we shall prove that applying the $f$-logarithm to them produces periods of $\chi_f$. As it stands,
    by \cite[Lemma 3.2]{Ber14}, we already know that
    \[
        t_f = \Log_f(x_0) \in \BdR^+.
    \]
    This series converges in $\BdR^+$ since $\Log_f$ converges on $\mathfrak{m}_{\mathbb{C}_p\!}$ and $\theta(x_n) \equiv \pi_{nr} \bmod (\pi_L)$ for all $n \geqslant 0$. Considering $\Log_f(x_n)$ for a different integer $n > 0$ would yield a~$\mathbb{Q}_p$-multiple of $t_f$. Note that $t_f$ depends not only on $f$ but also on the choice of the consistent sequence $\pi$ in Subsection \ref{S: dynamical-recovery-GL}. We proceed to show that it is a nonzero period of $\chi_f$.

    \begin{proposition}
        The character $\chi_f\!$ of $G_L\!$ is crystalline of Hodge-Tate weight $1$.
    \end{proposition}

    \begin{proof}
        The previous convergence in $\BdR^+$ is with respect to the $\ker(\theta)$-adic topology. So there exists some $m \geqslant 1$ for which $\theta(b_m x_0^m) = b_m\theta(x_0)^m = 0$, implying $\theta(x_0) = 0$, where $b_m \in K^\times$ is the $m$-th coefficient of the formal power series $\Log_f$. Therefore $\theta(t_f) = 0$ and hence $\chi_f$ has weight at least $1$. If it had weight $\geqslant  2$, the $\upsilon_\mathcal{R}$-valuation of $x_0 \bmod (\pi_L) = \varpi_0$ would be at least $2$, however
        \begin{align*}
            \upsilon_\mathcal{R}(\varpi_0)
            & = \lim_{m \to +\infty} q^m\upsilon_p(\pi_{mr})\\
            & = \lim_{m \to +\infty} p^{mr}\frac{1}{p^{mr-1}(p-1)}\\
            & = \frac{p}{p-1}
        \end{align*}
        is less than $2$ when $p$ is odd, a contradiction. For $p=2$, it suffices to compute the next component of $x_0$ and this leads to a contradiction. In particular $t_f$ is nonzero. For every $g \in G_L$, we have
        \[
            g.t_f = \Log_f(g.x_0) = \Log_f([\chi_f(g)]_f(x_0)) = \chi_f(g)t_f
        \]
        and hence $\chi_f$ is de Rham of Hodge-Tate weight $1$. To show that $\chi_f$ is crystalline, first observe that $t_f$ lies in $L \otimes_{L_0} \Bcris^+$ because $\theta(x_0) = 0$ and $\Log_f' \in \mathcal{O}_K[\![X]\!]$. This implies the existence of a crystalline period since
        \begin{align*}
            \dim_{\mathbb{Q}_p} \mathbb{Q}_p(\chi_f)
            & = \dim_{L} \Hom_{\mathbb{Q}_p[G_L]}(\mathbb{Q}_p(\chi_f), L \otimes_{L_0} \Bcris)\\
            & = \dim_{L} L \otimes_{L_0} \Hom_{\mathbb{Q}_p[G_L]}(\mathbb{Q}_p(\chi_f), \Bcris)\\
            & = \dim_{L_0} \D^*_\mathrm{cris}(\mathbb{Q}_p(\chi_f)).\qedhere
        \end{align*}
    \end{proof}

    \begin{remark}
        While in \cite{Spe18} the period $t_f$ is in $\Bcris^+$, this is no more the case here in general. A major obstruction arises from the fact that the formal power series $\Log_f$ is not, \textit{a priori}, defined over the unramified part of $K/\mathbb{Q}_p$ when $e > 1$.
    \end{remark}

    \subsection{Defining the Tate module structure}\label{S: define-Tate-module} Let us summarize the situation up to this point. Starting from $(f, u)$, we have, over a finite extension $L/K$, recovered a restriction of the action of $G_K$ on cells of $T_f$ via the action of the $\mathbb{Z}_p$-iterates of $u$. The chosen resulting $G_L$-character $\chi_f$ is such that $\mathbb{Q}_p(\chi_f)$ is crystalline of weight $1$, and thus encodes the data of a formal group. Nevertheless, this construction severs the direct link with the initial set $T_f$. Our final goal for this section is to relate the information carried by $\mathbb{Z}_p(\chi_f)$ back to the $G_L$-set $T_f$, thereby endowing it with a $\mathbb{Z}_p$-module structure that enables the recovery of the latent formal group.

        \subsubsection{A pivotal formal group} Since $\chi_f$ is crystalline of Hodge-Tate weight $1$, there exists a height-one formal group $H$ defined over $\mathcal{O}_L$ such that
        \[
            \T_p(H) \simeq \mathbb{Z}_p(\chi_f)
        \]
        as stated in Paragraph \ref{S: formal-groups-situation}. This formal group may not be the one we are looking for but, if Lubin's conjecture holds, it would be isomorphic to it. We will transport the structure of $\T_p(H)$ onto $T_f$. Any other choice of formal group $H'$ over $\mathcal{O}_L$ such that $\T_p(H') \simeq \mathbb{Z}_p(\chi_f)$ would induce the same structure on $T_f$. Since $\T_p(H)$ and $\mathbb{Z}_p(\chi_f)$ are isomorphic free $\mathbb{Z}_p[G_L]$-modules of rank $1$, we obtain:

        \begin{lemma}\label{R: Galois-action-TpH}
            The action of $G_L\!$ on $\T_p(H)\!$ is given by the character $\chi_f \colon G_L \rightarrow \mathbb{Z}_p^\times\!$.
        \end{lemma}

        We certainly wish $(f, u)$ to remain a pair of endomorphisms of $T_f$ after transporting the structure. To this end, we identify the unique endomorphisms of $H$ to which $f$ and $u$ should correspond by examining how they act on the constructed period $t_f$, which is a shared data. Indeed, $t_f$ was constructed from the dynamical system and by Lemma \ref{R: Galois-action-TpH} it is also a period of $\T_p(H)$. Let $s \in \{f, u\}$; therefore $s$ acts naturally on $T_f$. We want to define $s(t_f)$ arising from this action in a manner compatible with the construction of $t_f$. Recall that in \S\ref{S: regularity-GL} we have established this path in order to obtain a period:
        \[
            \begin{tikzcd}[row sep=2mm, column sep=5mm]
                & \Aty(\mathcal{O}_L) \ar[rr, hookrightarrow] && \AinfL \ar[rr, dashed, "\Log_f"] && L \otimes_{L_0} \Bcris^+\\
                \\
                & X_0 \ar[rr, mapsto]\ar[dddr, mapsto] && x_0 \ar[rr, mapsto]\ar[dddl, mapsto] && t_f\\
                T_f && \mathcal{R}\ar[from=uuur, twoheadrightarrow, crossing over]\ar[from=uuul, crossing over, "\phi_\varpi"] \ar[from=ll, crossing over, "\vartheta_f"]\\
                \\
                \pi \ar[rr, mapsto] && \varpi_0
            \end{tikzcd}
        \]
        where the dashed line is defined for all $x \in \AinfL \cap \ker(\theta)$. Considering maps in the same order as we did previously, $s(\pi)$ is sent to $(a^{\circ m}(s(\pi_{mr}) \bmod \pi_L))_{m \geqslant 0}$ which is~well defined in $\mathcal{R}$ because $s \bmod (\pi_L)$ commutes with $a$. This element is nothing but $s(\varpi_0)$ and is the image of $s(X_0)$ under $\phi_\varpi$. Hence these maps agree after action of $s$. By functoriality of the Witt vectors functor, $s(X_0)$ is sent to $s(x_0)$ and lifts $s(\varpi_0)$ modulo $(\pi_L)$. As $s(x_0) \in \ker(\theta)$, we let
        \[
            s(t_f) = \Log_f(s(x_0)) = s'(0)t_f.
        \]
        This definition is natural and coincides with the action of a formal group endomorphism on its period, as discussed in Paragraph \ref{S: formal-groups-situation}. It thus establishes how $f$ and $u$ act on $t_f$. From the perspective of $H$, we assert that $s'(0)t_f$ arises from $s_H$, the multiplication-by-$s'(0)$ endomorphism of $H$. Indeed, the composite of the functors considered in \S\ref{S: Fontaine-period-rings} is faithful:
        \[
            \begin{tikzcd}[row sep=11mm, column sep=13mm]
                & \REP_{\mathbb{Q}_p}^{\mathrm{cris},\mathrm{fgl}}(G_L) \ar[r, shift left=.6ex, "\D^*_\mathrm{cris}"]
                & \MF_L^{\mathrm{adm},\mathrm{fgl}}(\varphi) \ar[l, shift left=.6ex]
                \\
                \FGL/\mathcal{O}_L \ar[r, shift left=.6ex, "\T_p"]
                & \REP_{\mathbb{Z}_p}^{\mathrm{cris},\mathrm{fgl}}(G_L) \ar[l, shift left=.6ex]\ar[u, hookrightarrow, "\mathbb{Q}_p \otimes_{\mathbb{Z}_p} -"]
            \end{tikzcd}
        \]
        where the superscript ‘‘fgl'' stands for ‘‘arising from a formal group'', and $s_H$ acts in the same way on all de Rham periods of $\V_{\!p}(H)$, including the crystalline ones. We must therefore associate $f$ with $f_H$ when defining a bijection between $T_f$ and $\T_p(H)$. We also want this bijection to be $G_L$-equivariant. However, Lemma \ref{R: Galois-action-TpH} implies that the action of $G_L$ on $\T_p(H)$ is given by the one of the $\mathbb{Z}_p$-iterates of $u_H$, independently of the cell. Hence we must take into account the weak knowledge (\ref{fact-2}) in order to be compatible with the existing action of $G_L$ on $T_f$, although this does not guarantee that $u$ is associated to $u_H$ nor that it becomes an endomorphism of $T_f$ after transport of structure. Nevertheless, it suffices that $f$ be a formal group endomorphism for $u$ to be one as well.

        \subsubsection{Transport of structure} Since $H$ has height $1$, its group $H[p^n]$ of $p^n$-torsion points in $\mathfrak{m}_{\mathbb{C}_p\!}$ has exactly $p^n$ elements, and these points are also the roots of $f^{\circ n}_H$ because $f'(0)$ is a uniformizer in $\mathbb{Z}_p$. To define a bijection between $T_f$ and $\T_p(H)$ more easily, we reinterpret the latter Tate module by taking the transition maps to be $f_H$ instead of $[p]_H$:
        \[
            \T_p(H) = \varprojlim_{x \mapsto [p]_H(x)} H[p^n] \simeq \varprojlim_{x \mapsto f_H(x)} H[p^n]
        \]
        which is legitimate since $f'(0)$ is a uniformizer in $\mathbb{Z}_p$. Considering the above facts, it is clear that such a bijection exists. We now proceed to construct one satisfying our constraints.

        \begin{proposition}
            There exists a $\mathbb{Z}_p[G_L]$-module structure on the set $T_f$ for which it is a crystalline $G_L$-character of Hodge-Tate weight $1$ and $f \in \End_{\mathbb{Z}_p[G_L]}(T_f)$.
        \end{proposition}

        \begin{proof}
            Analogously to the constructions in \S\ref{S: dynamical-recovery-GL}, we define $S_H$ to be the subset of $\T_p(H)$ of elements with a nonzero second entry. Given $\beta = (\beta_n)_{n \geqslant 0} \in S_H$, let
            \[
                S_{H, \beta} = \{(y_n)_{n \geqslant 0} \in \T_p(H) \mid y_0 = \beta_0, \dots, y_\ell = \beta_\ell\} \subset S_H
            \]
            be a cell. In order to define a $G_L$-equivariant bijection $\tau \colon T_f \rightarrow \T_p(H)$ transporting $f$ to $f_H$, we begin by matching $S_f$ with $S_H$. Set $j = p^\ell - p^{\ell-1}$ and choose a system $(\alpha_1, \dots, \alpha_j)$ of elements of $S_f$ inducing a partition of $S_f$ as in \S\ref{S: dynamical-recovery-GL}. Similarly, pick a system $(\beta_1, \dots, \beta_j)$ of elements of $S_H$ where we run through the torsion points of $H$ of exact order $p^\ell$. Such a system, in turn, induces a partition of $S_H$. For every integers $i \in \{1,\dots,j\}$ and $n \geqslant 0$, and for all $g \in G_L$, set
            \begin{equation}\label{fact-3}
                \tau((f^{\circ n} \circ [\chi_{f, \alpha_i}(g)]_f)(\alpha_i)) = (f_H^{\circ n} \circ [\chi_f(g)]_H)(\beta_i)
            \end{equation}
            and $\tau(0) = 0$. Fixing $n = 0$ and taking $i \in \{1,\dots,j\}$ in (\ref{fact-3}) furnishes a well-defined bijection $\tau \colon S_{f, \alpha_i} \rightarrow S_{H, \beta_i}$ since the $\mathbb{Z}_p$-iterates of $u$ act simply transitively on $S_{f, \alpha_i}$ and the same holds about $u_H$ and $S_{f, \beta_i}$. Varying $n \geqslant 0$ in (\ref{fact-3}) allows to reach the remaining points of $T_f$ and $\T_p(H)$. Therefore (\ref{fact-3}) defines a $G_L$-equivariant bijection $\tau \colon T_f \rightarrow \T_p(H)$ associating $f$ with $f_H$, and the statement follows by transport of structure.
        \end{proof}

        \begin{remark}
            A different choice of representatives for the partitions in the previous proof would provide the same $\mathbb{Z}_p[G_L]$-module structure on $T_f$, but would permute the roles of the individual elements. This is immaterial for our purpose since what matters to us are the group law and the property that the power series $f$ remains an endomorphism of $T_f$ after transport of structure, which are unaffected by such choices. Nevertheless, the latent formal group $F$ obtained in Theorem \ref{R: main-theorem} will satisfy $\T_p(F) = T_f$ as $G_L$-sets, but only $\T_p(F) \simeq T_f$ as $\mathbb{Z}_p$-modules.
        \end{remark}

\section{The latent formal group}

\noindent Possessing a $p$-adic Tate module and having access to the existence of a quasi-inverse to the equivalence in Paragraph \ref{S: formal-groups-situation} is not sufficient to solve Lubin's conjecture: it is not enough to work up to isomorphism; one must have equality. For this purpose, we make sure that $f$ remains an endomorphism at each step of our constructions, until we obtain a formal group. Recall that $f$ acts on $T_f$ by componentwise evaluation or, equivalently, by multiplication by its linear coefficient, and all its other coefficients are uniquely determined by the action of $\mathbb{Z}_p$ on $T_f$. Our strategy requires an explicit approach in integral $p$-adic Hodge theory. One such approach proceeds through the work of Kisin in \cite{Kis06} and \cite{Kis09} to obtain a formal Breuil-Kisin module, then via the one of Zink in \cite{Zin01} and \cite{Zin02} on displays, indicated by the dashed path in the diagram:
\[
    \begin{tikzcd}[row sep=11mm, column sep=13mm]
        & \REP_{\mathbb{Q}_p}^{\mathrm{cris},\mathrm{fgl}}(G_L) \ar[r, shift left=.6ex, "\D^*_\mathrm{cris}"]
        & \MF_L^{\mathrm{adm},\mathrm{fgl}}(\varphi) \ar[l, shift left=.6ex]
        \\
        \FGL/\mathcal{O}_L \ar[r, shift left=.6ex, "\T_p"]\ar[d, shift left=.6ex]
        & \REP_{\mathbb{Z}_p}^{\mathrm{cris},\mathrm{fgl}}(G_L) \ar[u, hookrightarrow, "\mathbb{Q}_p \otimes_{\mathbb{Z}_p} -"] \ar[r, shift left=.6ex, dashed, "\D_\mathfrak{S}^*"] \ar[l, shift left=.6ex]
        & \BT^{\mathrm{f}}_\mathfrak{S}(\varphi) \ar[u, hookrightarrow, "D"]\ar[d, dashed, shift left=.6ex, "S \otimes_{\varphi, \mathfrak{S}} -"] \ar[l, shift left=.6ex]
        \\
        (p\textbf{-}\mathbf{Div}/\mathcal{O}_L)^\mathrm{c} \ar[u, dashed, shift left=.6ex, "\text{comult.}"] \ar[rr, shift left=.6ex]
        & & S\textbf{-}\mathbf{Win}/\mathcal{O}_L \ar[ll, shift left=.6ex, dashed, "\mathrm{BT}\,\circ\,\mathrm{Dsp}"]\ar[u, shift left=.6ex]
    \end{tikzcd}
\]
where the categories and functors appearing therein will be defined progressively throughout this section. At each stage, we track the action of $f$ on the object being constructed and ensure that it indeed encodes the data of its other coefficients just as $T_f$ does. These precautions are necessary in order to conclude that the formal group obtained in the end is the latent one.

    \subsection{The associated Breuil-Kisin module}\label{S: latent-Breuil-Kisin-module} Recall that we let $L_0$ be the maximal unramified extension of $\mathbb{Q}_p$ inside $L$ and let $\sigma$ be its Frobenius. We have also fixed a uniformizer $\pi_L$ of $L$, and write $E(X)$ for its Eisenstein minimal polynomial over $L_0$. Denote by $L_\infty\!$ the extension generated over $L$ by the components of a compatible sequence of $p$-th roots of $\pi_L$ in $\mathfrak{m}_{\mathbb{C}_p\!}$. We define two subrings $\mathfrak{S} \subset \mathfrak{S}^\mathrm{ur}$ of $\W(\mathcal{R})$ as follows. Let
    \[
        \mathfrak{S} = \W(\mathbb{F}_q)[\![X]\!]
    \]
    be the Kisin ring, endowed with a $\sigma$-semilinear Frobenius endomorphism $\varphi$ defined by $\varphi(X) = X^p$. According to \cite[\S2.1.1]{Kis06}, there exists an embedding $\mathfrak{S} \hookrightarrow \W(\mathcal{R})$ that preserves the Frobenius. Let $\mathscr{E}$ be the fraction field of the $p$-adic completion of $\mathfrak{S}[1/X]$. The previous embedding extends to an embedding $\mathscr{E} \hookrightarrow \W(\Frac(\mathcal{R}))$, and we denote by $\mathscr{E}^\mathrm{ur}$ the maximal unramified extension of $\mathscr{E}$ inside $\W(\Frac(\mathcal{R}))[1/p]$ and by $\mathcal{O}_{\mathscr{E}^\mathrm{ur}}$ its ring of integers. We then define
    \[
        \mathfrak{S}^\mathrm{ur} = \mathcal{O}_{\mathscr{E}^\mathrm{ur}} \cap \W(\mathcal{R}).
    \]
    Since the absolute Galois group of $L_\infty$ fixes $\mathfrak{S}$, it acts naturally on $\mathfrak{S}^\mathrm{ur}$.

        \subsubsection{Breuil-Kisin modules} A Breuil-Kisin module $(\mathfrak{M}, \varphi)$ is a finite free $\mathfrak{S}$-module $\mathfrak{M}$ equipped with a $\varphi_\mathfrak{S}$-semilinear Frobenius endomorphism $\varphi \colon \mathfrak{M} \rightarrow \mathfrak{M}$ whose cokernel of its $\mathfrak{S}$-linearization $1 \otimes \varphi \colon \mathfrak{S} \otimes_{\mathfrak{S}, \varphi} \mathfrak{M} \rightarrow \mathfrak{M}$ is killed by $E(X)$. Together with $\varphi$-equivariant $\mathfrak{S}$-linear maps, these objects form the category $\BT_\mathfrak{S}(\varphi)$. Kisin obtained in \cite{Kis06} a functor $\D_\mathfrak{S}^*$ such that the following diagram commutes:
        \[
            \begin{tikzcd}[column sep=11mm, row sep=13mm]
                \REP_{\mathbb{Q}_p}^{\mathrm{cris},\{0,1\}}(G_L) \ar[r, "\D^*_\mathrm{cris}", shift left=.6ex]
                & \MF^{\mathrm{adm},\{0,1\}}_L(\varphi) \ar[l, shift left=.6ex]
                \\
                \REP_{\mathbb{Z}_p}^{\mathrm{cris},\{0,1\}}(G_L) \ar[r, shift left=.6ex, "\D_\mathfrak{S}^*"]\ar[u, hookrightarrow, "\mathbb{Q}_p \otimes_{\mathbb{Z}_p} -"]
                & \BT_\mathfrak{S}(\varphi) \ar[l, shift left=.6ex]\ar[u, hookrightarrow, "D"]
            \end{tikzcd}
        \]
        where $D$ is the functor defined in \cite[\S1]{Kis06} while ‘‘$\{0,1\}$'' stands for the weights or the jumps in $\{0,1\}$. The full subcategory of Breuil-Kisin modules corresponding to the representations arising from $\FGL/\mathcal{O}_L$ is denoted by $\BT_\mathfrak{S}^\mathrm{f}(\varphi)$, and its objects are called formal Breuil-Kisin modules.

        \subsubsection{Construction} Set
        \[
            M_f = \Hom_{\mathbb{Z}_p[G_{L_\infty}]}(T_f, \mathfrak{S}^\mathrm{ur})
        \]
        to be the free $\mathfrak{S}$-module of rank $1$ appearing in \cite[Lemma 2.1.10]{Kis06}. The action of $f$ on $T_f$ extends to $M_f$ to be the multiplication by its linear coefficient. Indeed, given a generator $\pi \mapsto s$ of $M_f$, one has $f(\pi \mapsto s) = f'(0)(\pi \mapsto s)$ because $f(\pi) = f'(0)\pi$. We do not loose any information on the remaining coefficients of $f$ since they are encoded in the action of $\mathbb{Z}_p \subset \mathfrak{S}$ on $M_f$. We let
        \[
            \mathfrak{M}_f = \D_\mathfrak{S}^*(T_f)
        \]
        which is a rank-one formal Breuil-Kisin module. By \cite[Lemma 2.1.10]{Kis06}, we may assume that $\mathfrak{M}_f \subseteq M_f$ as $\mathfrak{S}$-modules. Therefore the action of $f$ on $M_f$ restricts to $\mathfrak{M}_f$ and the inclusion guarantees that the latter encapsulates the other coefficients of the formal power series $f$.

    \subsection{The associated window}\label{S: latent-window} Let $S$ be the $p$-adic completion of the PD-envelope of $\W(\mathbb{F}_q)[X]$ with respect to the ideal $(E(X))$, and let $\Fil^1 S \subset S$ be the closure of the ideal generated by the divided powers of $E(X)$. As topological rings, one has an identification $S/\Fil^1 S \simeq \mathcal{O}_L$. By definition, one has $\mathfrak{S} \subset S$ and the Frobenius $\varphi$ on $\mathfrak{S}$ extends to a $\sigma$-semilinear Frobenius endomorphism $\varphi$ on the ring $S$.

        \subsubsection{\texorpdfstring{$S$}{S}-windows} Let $R \in \{\mathcal{O}_L/(p^n)\,;\, n \geqslant 0\}$ be a ring. An $S$-window $(W, W', \varphi)$ over $R$ is a finite projective $S$-module $W$ together with a submodule $W'$ containing $(\Fil^1 S)W$ and a $\sigma$-semilinear Frobenius endomorphism $\varphi \colon W \rightarrow W$ such that:
        \begin{enumerate}
            \item $W/W'$ is a projective $R$-module,
            \item $\varphi(W') \subseteq pW$ and $W$ is generated by $\varphi(W) + \varphi/p(W')$ as an $S$-module,
            \item since $pW \subseteq (1 \otimes \varphi)(\varphi^*(W))$, where $\varphi^*(W) = S \otimes_{S, \varphi} W$, the map $1 \otimes \varphi$ is injective and there exists a unique $S$-linear map $\psi \colon W \rightarrow \varphi^*(W)$ such that $(1 \otimes \varphi) \circ \psi \colon W \rightarrow W$ is the multiplication-by-$p$ map. Then, for $n$ sufficiently large,
            \[
                \varphi^{n-1*}(\psi) \circ \varphi^{n-2*}(\psi) \circ \cdots \circ \psi \colon W \longrightarrow \varphi^{n*}(W)
            \]
            factors through $(p, \Fil^1 S)\varphi^{n*}(W)$.
        \end{enumerate}
        Together with $\varphi$-equivariant $S$-linear maps preserving the distinguished submodule, these objects form the category $S\textbf{-}\mathbf{Win}/R$. In the particular case of $R = \mathcal{O}_L$, the data of an $S$-window over $\mathcal{O}_L$ is equivalent to the data of a system $(W_n)_{n \geqslant 1}$, where $W_n$ is an $S$-window over $\mathcal{O}_L/(p^n)$ and is obtained from $W_{n+1}$ by base change for every integer $n \geqslant 1$. The extension
        \[
            S \otimes_{\varphi, \mathfrak{S}} - \colon \BT_\mathfrak{S}^\mathrm{f}(\varphi) \longrightarrow S\textbf{-}\mathbf{Win}/\mathcal{O}_L
        \]
        defined in \cite[\S1.2.3]{Kis09} is an equivalence of categories.

        \subsubsection{Construction} Set
        \[
            W_f = S \otimes_{\varphi, \mathfrak{S}} \mathfrak{M}_f
        \]
        to be the $S$-window over $\mathcal{O}_L$ associated to our formal Breuil-Kisin module. Since Zink's functors in \cite{Zin01} for constructing $p$-divisible groups are formulated in terms of the limit description, we shall work with this perspective. Hence let $(W_{f,n})_{n \geqslant 1}$ be the corresponding system. Upon passing to the quotient of $W_f$ with which $W_{f,n}$ is identified, we may assume that every $W_{f,n}$ is a quotient of $W_f$. The action of $f$ on  $\mathfrak{M}_f$ extends to an action of $f$ on $W_f$ to be the multiplication-by-$f'(0)$ map. This action descends compatibly to each quotient $W_{f,n}$. Taking the $W_{f,n}$ as quotients of $W_f$ guarantees that the remaining coefficients of $f$ are uniquely determined by the action of $\mathbb{Z}_p \subset S$ on each of these objects.

    \subsection{The associated \texorpdfstring{$p$}{p}-divisible group}\label{S: latent-p-divisible-group} Denote by $(p\textbf{-}\mathbf{Div}/\mathcal{O}_L)^\mathrm{c}$ the category of finite-height connected $p$-divisible groups over $\mathcal{O}_L$. In \cite{Zin01}, Zink introduced two functors, namely $\Dsp$ and $\bt$, such that
    \[
        \bt \circ \Dsp \colon S\textbf{-}\mathbf{Win}/\mathcal{O}_L \longrightarrow (p\textbf{-}\mathbf{Div}/\mathcal{O}_L)^\mathrm{c}
    \]
    is an equivalence. These functors are defined on $S$-windows over $\mathcal{O}_L$ expressed as a limit. Let $n \geqslant 1$ and $R_n = \mathcal{O}_L/(p^n)$. Set $P_{f,n} = \Dsp(W_{f,n}) = \W(R_n) \otimes_{\varkappa_n, S} W_{f,n}$ which is a $\W(R_n)$-module, with $\varkappa_n \colon S \rightarrow \W(R_n)$ the map of \cite[p. 499]{Zin01}. The action of $f$ on $W_{f,n}$ extends to an action on $P_{f,n}$ and, again, is the multiplication-by-its-linear-coefficient map. All the other coefficients of $f$ are uniquely determined by the action of $\mathbb{Z}_p \subseteq \W(R_n)$ on $P_{f,n}$. Then we set
    \[
        \Gamma_f = \bt(\Dsp(W_f)) = \varprojlim_{m \geqslant 1} \bt(P_{f,m})
    \]
    which is a height-one connected $p$-divisible group over $\mathcal{O}_L$. Let $\W_\mathrm{f}$ be the functor of Witt vectors with finite support, one defines $\Gamma_{f,n} = \bt(P_{f,n})$ as a certain quotient of the functor $\W_\mathrm{f}(-) \otimes_{\W(R_n)} P_{f,n}$. The action of $f$ on $P_{f,n}$ extends to the previous functor, descends to the quotient $\Gamma_{f,n}$ and is compatible with the transition maps. Consequently the action of $f$ passes to the limit $\Gamma_f$ and is the multiplication-by-$f'(0)$ map. The other coefficients of $f$ are uniquely determined by the action of $\mathbb{Z}_p$, that is, by the structural data of $\Gamma_f$ arising from the structure of $T_f$.

    \subsection{The main theorem}\label{S: latent-formal-group} We have thus constructed $\Gamma_f$, a height-one connected $p$-divisible group defined over $\mathcal{O}_L$, whose followed path is summarized below:
    \[
        \begin{tikzcd}[column sep=11mm, row sep=13mm]
            \REP_{\mathbb{Z}_p}^{\mathrm{cris},\mathrm{fgl}}(G_L) \ar[r] & \BT^{\mathrm{f}}_\mathfrak{S}(\varphi) \ar[d] & T_f \ar[r, mapsto] & \mathfrak{M}_f \ar[d, mapsto]\\
            (p\textbf{-}\mathbf{Div}/\mathcal{O}_L)^\mathrm{c} & S\textbf{-}\mathbf{Win}/\mathcal{O}_L \ar[l] & \Gamma_f & W_f \ar[l, mapsto]
        \end{tikzcd}
    \]
    and ensure that $f$ is an endomorphism of all of these objects. It follows that Lubin's conjecture holds in our setting.

    \begin{theorem}\label{R: main-theorem}
        There exists a formal group $F$ over $\mathcal{O}_K\!$ such that $f, u \in \End_{\mathcal{O}_K}(F)$.
    \end{theorem}

    \begin{proof}
        The group law of $\Gamma_f$ induces a comultiplication on its Hopf algebra; this is a map $\mathcal{O}_L[\![T]\!] \rightarrow \mathcal{O}_L[\![X,Y]\!]$ given by sending $T$ to $F(X,Y)$ a height-one formal group defined over $\mathcal{O}_L$. By construction $f$ is an endomorphism of $\Gamma_f$ and therefore is an endomorphism of $F$. By the uniqueness property stated in \S\ref{S: logarithm}, we deduce that $F$ is defined over $\mathcal{O}_L \cap K = \mathcal{O}_K$. Finally, as $u$ is defined over $\mathcal{O}_K$ and commutes with $f$, it is necessarily an $\mathcal{O}_K$-endomorphism of $F$ as well.
    \end{proof}

\end{document}